\documentclass[11pt]{article}
\usepackage{amsfonts,amsmath,bm, xspace,caption,amssymb}
\usepackage{multirow,graphicx, algorithm,algorithmic,rotating}
\usepackage{url,cite}
\usepackage{fullpage}

\newtheorem{theorem}{Theorem}[section]
\newtheorem{lemma}[theorem]{Lemma}

\newtheorem{definition}[theorem]{Definition}

\def \endprf{\hfill {\vrule height6pt width6pt depth0pt}\medskip}
\newenvironment{proof}{\noindent {\bf Proof} }{\endprf\par}

\newcommand{\eat}[1]{}

\newcommand{\name}{\textsc{Hottopixx}\xspace}

\newcommand{\R}{\mathbb{R}}

\newcommand{\vct}[1]{\bm{#1}}
\newcommand{\mtx}[1]{\bm{#1}}

\newcommand{\diag}{\operatorname{diag}}
\newcommand{\trace}{\operatorname{Tr}}

\newcommand{\norm}[1]{\big\lVert{#1}\big\rVert}

\newcommand{\ionorm}[1]{\norm{#1}_{\infty,1}}

\newcommand{\minimize}{\operatorname*{minimize}}

\newcommand{\st}{\mbox{subject to}}

\newcommand{\eq}[1]{(\ref{eq:#1})}

\title{Factoring nonnegative matrices with linear programs}

\author{
Victor Bittorf$^*$, Benjamin Recht$^*$, Christopher R\'{e}$^*$, and Joel A.~Tropp$^\dagger$\\
$^*$ Computer Sciences Department, University of Wisconsin-Madison\\
$^\dagger$ Computational and Mathematical Sciences, California Institute of Technology}

\date{June 2012.  Last Revised Feb.~2013.}

\begin{document}

\maketitle

\vspace{-0.3in}

\begin{abstract}
This paper describes a new approach, based on linear programming, for computing nonnegative matrix factorizations (NMFs).  The key idea is a data-driven model for the factorization where the most salient features in the data are used to express the remaining features.  More precisely, given a data matrix $\mtx{X}$, the algorithm identifies a matrix $\mtx{C}$ that satisfies $\mtx{X} \approx \mtx{CX}$ and some linear constraints.  The constraints are chosen to ensure that the matrix $\mtx{C}$ selects features; these features can then be used to find a low-rank NMF of $\mtx{X}$.  A theoretical analysis demonstrates that this approach has guarantees similar to those of the recent NMF algorithm of Arora et al.~(2012).  In contrast with this earlier work, the proposed method extends to more general noise models and leads to efficient, scalable algorithms.  Experiments with synthetic and real datasets provide evidence that the new approach is also superior in practice.  An optimized C++ implementation can factor a multigigabyte matrix in a matter of minutes.
\end{abstract}


\section{Introduction}

Nonnegative matrix factorization (NMF) is a popular approach
for selecting features in data~\cite{LeeSeungNMFNature,LeeNMFAlgorithms,Hofmann99,Smaragdis03}. \eat{\bf [xxx add cites! xxx]}
Many machine-learning and data-mining software packages
(including Matlab~\cite{MatlabNMFurl}, R~\cite{Gaujoux10}, and Oracle Data Mining~\cite{OracleDataMiningNMFurl})
now include heuristic computational methods for NMF.
Nevertheless, we still have limited theoretical understanding
of when these heuristics are correct.

The difficulty in developing rigorous methods for NMF
stems from the fact that the problem is computationally
challenging.  Indeed, Vavasis has shown that NMF is \textsf{NP}-Hard~\cite{Vavasis07};
see~\cite{Arora12} for further worst-case hardness results.
As a consequence, we must instate additional assumptions on
the data if we hope to compute nonnegative matrix
factorizations in practice.

In this spirit,  Arora, Ge, Kannan, and Moitra (AGKM) 
have exhibited a polynomial-time algorithm for NMF that is 
provably correct---provided that the data is drawn from an appropriate
model, based on ideas from~\cite{DonohoStodden03}.
The AGKM result describes one circumstance where we
can be sure that NMF algorithms are capable of producing meaningful
answers.  This work has the potential to make an
impact in machine learning because proper feature selection
is an important preprocessing step for many other techniques.
Even so, the actual impact is damped by
the fact that the AGKM algorithm is too
computationally expensive for large-scale problems and is not tolerant to departures from
the modeling assumptions.   Thus, for NMF, there remains a gap
between the theoretical exercise and the actual practice
of machine learning.

The present work presents a scalable, robust algorithm
that can successfully solve the NMF problem under appropriate
hypotheses.  Our first contribution is a new formulation
of the nonnegative feature selection problem that only requires
the solution of a single linear program.  Second, we provide a theoretical analysis of this algorithm.
This argument shows that our method
succeeds under the same modeling assumptions as the AGKM algorithm with an additional~\emph{margin constraint} that is common in machine learning.   We prove that if there exists a unique, well-defined model, then we can recover this model accurately; our error bound improves substantially on the error bound for the AGKM algorithm in the high SNR regime.  One may argue that NMF only ``makes sense'' (i.e., is well posed) when a unique solution exists, and so we believe our result has independent interest.  Furthermore, our algorithm can be adapted for a wide class of noise models.

In addition to these theoretical contributions, our work also includes
a major algorithmic and experimental component.  Our formulation of NMF
allows us to exploit methods from
operations research and database systems to design solvers
that scale to extremely large datasets.  We develop an
efficient stochastic gradient descent (SGD) algorithm that is (at least)
two orders of magnitude faster than the approach of AGKM when both are implemented in Matlab.
We describe a parallel implementation of our SGD algorithm that
can robustly factor matrices with $10^5$ features
and $10^6$ examples in a few minutes on a multicore workstation.

Our formulation of NMF uses a data-driven modeling approach to
simplify the factorization problem.  More precisely, we search
for a small collection of rows from the data matrix that can be
used to express the other rows.  This type of approach appears
in a number of other factorization problems, including
rank-revealing QR~\cite{Gu96}, interpolative
decomposition~\cite{MahoneyDrineas09}, subspace
clustering~\cite{Elhamifar09,Soltanolkotabi11},
dictionary learning~\cite{Esser12}, and others.
Our computational techniques can be adapted to address
large-scale instances of these problems as well.

\section{Separable Nonnegative Matrix Factorizations and Hott Topics}

{\bf Notation.}  For a matrix $\mtx{M}$ and indices $i$ and $j$, we write $\mtx{M}_{i \cdot}$ for the $i$th row of $\mtx{M}$ and $\mtx{M}_{\cdot j}$ for the $j$th column of $\mtx{M}$.  We write $M_{ij}$ for the $(i, j)$ entry.

Let $\mtx{Y}$ be a nonnegative $f\times n$ data matrix with columns indexing examples and rows indexing features.  Exact NMF seeks a factorization $\mtx{Y} = \mtx{F}\mtx{W}$ where the feature matrix $\mtx{F}$ is $f\times r$, where the weight matrix $\mtx{W}$ is $r\times n$, and both factors are nonnegative.  Typically, $r \ll \min\{f,n\}$.

Unless stated otherwise, we assume that each row of the data matrix $\mtx{Y}$ is normalized so it sums to one.  Under this hypothesis, we may also assume that each row of $\mtx{F}$ and of $\mtx{W}$ also sums to one~\cite{Arora12}.

It is notoriously difficult to solve the NMF problem. 
Vavasis showed that it is {\sf NP}-complete to decide whether a matrix
admits a rank-$r$ nonnegative factorization~\cite{Vavasis07}.
AGKM proved that an exact NMF algorithm can be used to solve
{\sf 3-SAT} in subexponential time~\cite{Arora12}.

The literature contains some mathematical analysis of NMF that can
be used to motivate algorithmic development.
Thomas~\cite{Thomas74} developed a necessary and
sufficient condition for the existence of a rank-$r$ NMF.
More recently, Donoho and Stodden~\cite{DonohoStodden03}
obtained a related sufficient condition for uniqueness.
AGKM exhibited an algorithm that
can produce a nonnegative matrix factorization under
a weaker sufficient condition.  To state their results, we need a definition.

\begin{definition}
A set of vectors $\{ \vct{v}_1, \ldots, \vct{v}_r \} \subset \R^d$ is \emph{simplicial} if no vector $\vct{v}_i$ lies in the convex hull of $\{ \vct{v}_j : j \neq i \}$. The set of vectors is \emph{$\alpha$-robust simplicial} if, for each $i$, the $\ell_1$ distance from $\vct{v}_i$ to the convex hull of $\{ \vct{v}_j : j \neq i \}$ is at least $\alpha$.
Figure~\ref{fig:robust-simplicial} illustrates these concepts.
\end{definition}

\begin{figure}
\begin{minipage}{.48\textwidth}
  \nointerlineskip \vspace{-3\baselineskip}%
  \hspace{\fill}\rule{\linewidth}{.7pt}\hspace{\fill}%
  \par\nointerlineskip 
\captionof{algorithm} {AGKM: Approximably Separable Nonnegative Matrix Factorization~\cite{Arora12}}\label{alg:arora}
  \nointerlineskip \vspace{-.5\baselineskip}%
  \hspace{\fill}\rule{\linewidth}{.7pt}\hspace{\fill}%
  \par\nointerlineskip 
\begin{algorithmic}[1]
  \STATE Initialize $R = \emptyset$.
	\STATE Compute the $f\times f$ matrix $\mtx{D}$ with $D_{ij}= \| \mtx{X}_{i\cdot}-\mtx{X}_{j\cdot}\|_1$.
	\FOR{$k=1,\ldots f$}
		\STATE  Find the set $\mathcal{N}_k$ of rows that are at least $5\epsilon/\alpha + 2 \epsilon$ away from $\mtx{X}_{k\cdot}$.
		\STATE Compute the distance $\delta_k$ of $\mtx{X}_{k\cdot}$ from $\operatorname{conv}( \{\mtx{X}_{j\cdot}~:~j\in\mathcal{N}_k\})$.
		\STATE {\bf if} $\delta_k > 2 \epsilon$, add $k$ to the set $R$.
	\ENDFOR 
\STATE Cluster the rows in $R$ as follows: $j$ and $k$ are in the same cluster if $D_{jk} \leq 10 \epsilon/\alpha+ 6 \epsilon$.
\STATE Choose one element from each cluster to yield $\mtx{W}$.
\STATE $\mtx{F} = \arg\min_{\mtx{Z}\in \R^{f\times r}} \ionorm{\mtx{X} - \mtx{Z} \mtx{W}}$
\end{algorithmic}
  \nointerlineskip \vspace{-0.125\baselineskip}%
  \hspace{\fill}\rule{\linewidth}{.7pt}\hspace{\fill}%
  \par\nointerlineskip 
    \end{minipage}
\hspace{0.02\textwidth}
\begin{minipage}[t]{.48\textwidth}
  \centering
  \includegraphics[width=2.5in]{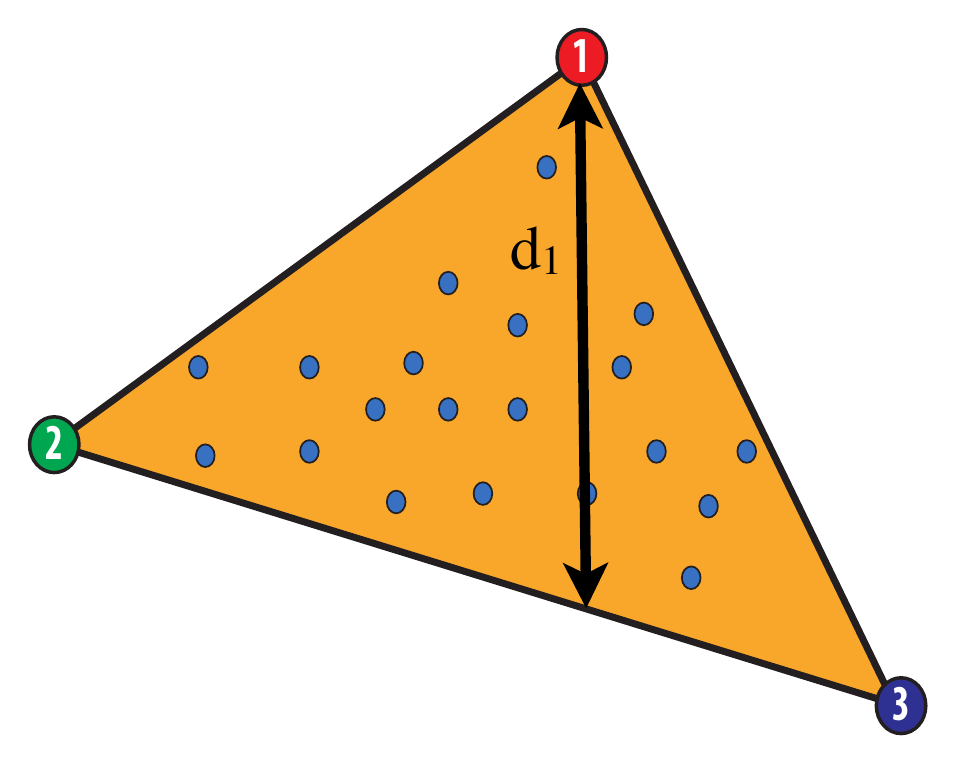}
\captionof{figure}{\small \label{fig:robust-simplicial}  Numbered circles are hott topics.  Their convex hull (orange) contains the other topics (small circles), so the data admits a separable NMF.  The arrow $d_1$ marks the $\ell_1$ distance from hott topic (1) to the convex hull of the other two hott topics; definitions of $d_2$ and $d_3$ are similar. The hott topics are $\alpha$-robustly simplicial when each $d_i \geq \alpha$.}\label{fig:robust-simp} 
\end{minipage}
\end{figure}

These ideas support the uniqueness results of Donoho and Stodden and the AGKM algorithm. Indeed, we can find an NMF of $\mtx{Y}$ efficiently if $\mtx{Y}$ contains a set of $r$ rows that is simplicial and whose convex hull contains the remaining rows.

\begin{definition}
An NMF $\mtx{Y}=\mtx{F}\mtx{W}$ is called~\emph{separable} if the rows of $\mtx{W}$ are simplicial and there is a permutation matrix $\mtx{\Pi}$ such that
\begin{equation}\label{eq:F-form}
	\mtx{\Pi}\mtx{F} = \left[\begin{array}{c} \mathbf{I}_r\\ \mtx{M} \end{array}\right]\,.
\end{equation}
\end{definition}

To compute a separable factorization of $\mtx{Y}$, we must first identify a simplicial set of rows from $\mtx{Y}$.  Afterward, we compute weights that express the remaining rows as convex combinations of this distinguished set.  We call the simplicial rows \emph{hott} and the corresponding features \emph{hott topics}.

This model allows us to express all the features for a particular instance if we know the values of the instance at the simplicial rows.  This assumption can be justified in a variety of applications.  For example, in text, knowledge of a few keywords may be sufficient to reconstruct counts of the other words in a document.  In vision, localized features can be used to predict gestures.  In audio data, a few bins of the spectrogram may allow us to reconstruct the remaining bins.  

While a nonnegative matrix one encounters in practice might not admit a separable factorization, it may be \emph{well-approximated} by a nonnnegative matrix with separable factorization.  AGKM derived an algorithm for nonnegative matrix factorization of a matrix that is well-approximated by a separable factorization.  To state their result,  we introduce a norm on $f \times n$ matrices:
\[
	\ionorm{\mtx{\Delta}} := \max_{1\leq i\leq f} \sum_{j=1}^n |\Delta_{ij}| \,.
\]
\begin{theorem}[AGKM~\cite{Arora12}]
Let $\epsilon$ and $\alpha$ be nonnegative constants satisfying $\epsilon \leq \frac{\alpha^2}{20+13\alpha}$.  Let $\mtx{X}$ be a nonnegative data matrix.
Assume $\mtx{X}=\mtx{Y}+\mtx{\Delta}$ where $\mtx{Y}$ is a nonnegative matrix whose rows have unit $\ell_1$ norm, where $\mtx{Y} = \mtx{FW}$ is a rank-$r$ separable factorization in which the rows of $\mtx{W}$ are $\alpha$-robust simplicial, and where $\ionorm{\mtx{\Delta}}\leq \epsilon$.  Then Algorithm~\ref{alg:arora} finds a rank-$r$ nonnegative factorization $\hat{\mtx{F}}\hat{\mtx{W}}$ that satisfies the error bound $\ionorm{\mtx{X}-\hat{\mtx{F}}\hat{\mtx{W}}} \leq 10 \epsilon/\alpha+7\epsilon$.
\end{theorem}

In particular, the AGKM algorithm computes the factorization exactly when $\epsilon=0$.   Although this method is guaranteed to run in polynomial time, it has many undesirable features.  First, the algorithm requires a priori knowledge of the parameters $\alpha$ and $\epsilon$.  It may be possible to calculate $\epsilon$, but we can only estimate $\alpha$ if we know which rows are hott.  Second, the algorithm computes all $\ell_1$ distances between rows at a cost of $O(f^2 n)$.  Third, for every row in the matrix, we must determine its distance to the convex hull of the rows that lie at a sufficient distance; this step requires us to solve a linear program for each row of the matrix at a cost of $\Omega(fn)$.  Finally, this method is intimately linked to the choice of the error norm $\ionorm{\cdot}$.  It is not obvious how to adapt the algorithm for other noise models.  We present a new approach, based on linear programming, that overcomes these drawbacks.

\section{Main Theoretical Results: NMF by Linear Programming}

This paper shows that we can factor an approximately separable nonnegative matrix by solving a linear program.  A major advantage of this formulation is that
it scales to very large data sets.

Here is the key observation:  Suppose that $\mtx{Y}$ is any $f \times n$ nonnegative matrix that admits a rank-$r$ separable factorization $\mtx{Y} = \mtx{FW}$.   If we pad $\mtx{F}$ with zeros to form an $f \times f$ matrix, we have
\[
\mtx{Y} =  \mtx{\Pi}^T\left[\begin{array}{cc} \mathbf{I}_r & \mtx{0} \\ \mtx{M} & \mtx{0} \end{array}\right] \mtx{\Pi} \mtx{Y} =: \mtx{C}\mtx{Y}
\]
We call the matrix $\mtx{C}$ \emph{factorization localizing}.  Note that any factorization localizing matrix $\mtx{C}$ is an element of the polyhedral set
\[
\Phi(\mtx{Y}):=\{\mtx{C}\geq \mtx{0}~:~  \mtx{CY}=\mtx{Y},\ \trace(\mtx{C}) = r,\ C_{jj} \leq 1 \ \forall j,\ C_{ij}\leq C_{jj} \ \forall i,j\}.
\]

Thus, to find an exact NMF of $\mtx{Y}$, it suffices to find a feasible element of $\mtx{C}\in \Phi(\mtx{Y})$ whose diagonal is integral.  This task can be accomplished by linear programming.  Once we have such a $\mtx{C}$, we construct $\mtx{W}$ by extracting the rows of $\mtx{X}$ that correspond to the indices $i$ where $C_{ii}=1$.  We construct the feature matrix $\mtx{F}$ by extracting the nonzero columns of $\mtx{C}$.  This approach is summarized in Algorithm~\ref{alg:lp-noiseless}.  In turn, we can prove the following result.

\begin{theorem} \label{prop:noiseless-case}
Suppose $\mtx{Y}$ is a nonnegative matrix with a rank-$r$ separable factorization $\mtx{Y} = \mtx{F}\mtx{W}$. Then Algorithm~\ref{alg:lp-noiseless} constructs a rank-$r$ nonnegative matrix factorization of $\mtx{Y}$.
\end{theorem}

As the theorem suggests, we can isolate the rows of $\mtx{Y}$ that yield a simplicial factorization  by solving a single linear program.  The factor $\mtx{F}$ can be found by extracting columns of $\mtx{C}$.

\begin{algorithm}[t]
  \caption{Separable Nonnegative Matrix Factorization by Linear Programming}
   \begin{algorithmic}[1]
  \REQUIRE An $f\times n$ nonnegative matrix $\mtx{Y}$ with a rank-$r$ separable NMF.
  \ENSURE An $f\times r$ matrix $\mtx{F}$ and $r\times n$  matrix $\mtx{W}$ with $\mtx{F}\geq \mtx{0}$, $\mtx{W}\geq \mtx{0}$, and $\mtx{Y}=\mtx{F}\mtx{W}$.
  \STATE Find the unique $\mtx{C} \in \Phi(\mtx{Y})$ to minimize $\vct{p}^T \diag(\mtx{C})$ where $\vct{p}$ is any vector with distinct values.
  \STATE Let $I=\{i~:~C_{ii}=1\}$ and set $\mtx{W}=\mtx{Y}_{I\cdot}$ and $\mtx{F}=\mtx{C}_{\cdot I}$.
\end{algorithmic} \label{alg:lp-noiseless}
\end{algorithm}

\subsection{Robustness to Noise}\label{sec:robustness}

Suppose we observe a nonnegative matrix $\mtx{X}$ whose rows sum to one.  Assume that $\mtx{X} = \mtx{Y} + \mtx{\Delta}$ where $\mtx{Y}$ is a nonnegative matrix whose rows sum to one, which has a rank-$r$ separable factorization $\mtx{Y} = \mtx{FW}$ such that the rows of $\mtx{W}$ are $\alpha$-robust simplicial, and where $\ionorm{\mtx{\Delta}} \leq \epsilon$.  Define the polyhedral set
\[
\Phi_\tau(\mtx{X}):=\left\{\mtx{C}\geq \mtx{0}~:~  \ionorm{\mtx{C}\mtx{X}-\mtx{X}}\leq \tau, \trace(\mtx{C}) = r, C_{jj} \leq 1 \ \forall j, C_{ij}\leq C_{jj}\ \forall i,j\right\}
\]
The set $\Phi(\mtx{X})$ consists of matrices $\mtx{C}$ that \emph{approximately} locate a factorization of $\mtx{X}$.  We can prove the following result.

\begin{theorem}\label{prop:noisy-case}
Suppose that $\mtx{X}$ satisfies the assumptions stated in the previous paragraph.
Furthermore, assume that for every row $\mtx{Y}_{j,\cdot}$ that is not hott, we have the margin constraint $\|\mtx{Y}_{j,\cdot} - \mtx{Y}_{i,\cdot}\|\geq d_0$ for all hott rows $i$.  Then we can find a nonnegative factorization satisfying $\ionorm{\mtx{X}-\hat{\mtx{F}}\hat{\mtx{W}}}\leq 2\epsilon$ provided that $\epsilon <\tfrac{ \min\{ \alpha d_0, \alpha^2\} }{9(r+1)}$.  Furthermore, this factorization correctly identifies the hott topics appearing in the separable factorization of $\mtx{Y}$.
\end{theorem}

Algorithm~\ref{alg:lp-noisycase} requires the solution of two linear programs.  The first minimizes a cost vector over $\Phi_{2\epsilon}(\mtx{X})$.  This lets us find $\hat{\mtx{W}}$.  Afterward, the matrix $\hat{\mtx{F}}$ can be found by setting
\begin{equation}\label{eq:cleaning}
\hat{\mtx{F}} = \arg\min_{\mtx{Z}\geq \mtx{0}} ~ \ionorm{\mtx{X} - \mtx{Z}\hat{\mtx{W}}} \,.
\end{equation}

\begin{algorithm}[t]
  \caption{Approximably Separable Nonnegative Matrix Factorization by Linear Programming}
   \begin{algorithmic}[1]
  \REQUIRE An $f\times n$ nonnegative matrix $\mtx{X}$ that satisfies the hypotheses of Theorem~\ref{prop:noisy-case}.
  \ENSURE An $f\times r$ matrix $\mtx{F}$ and $r\times n$  matrix $\mtx{W}$ with $\mtx{F}\geq \mtx{0}$, $\mtx{W}\geq \mtx{0}$, and $\ionorm{\mtx{X}-\mtx{F}\mtx{W}}\leq 2\epsilon$.
  \STATE Find $\mtx{C} \in \Phi_{2\epsilon}(\mtx{X})$ that minimizes $\vct{p}^T \diag{\mtx{C}}$ where $\vct{p}$ is any vector with distinct values.
  \STATE Let $I=\{i~:~C_{ii}=1\}$ and set $\mtx{W}=\mtx{X}_{I\cdot}$.
\STATE Set $\mtx{F} = \arg\min_{\mtx{Z}\in \R^{f\times r}} \ionorm{\mtx{X} - \mtx{Z} \mtx{W}}$
\end{algorithmic} \label{alg:lp-noisycase}
\end{algorithm}

Our robustness result requires a \emph{margin-type} constraint assuming that the original configuration consists either of duplicate hott topics, or topics that are reasonably far away from the hott topics.  On the other hand, under such a margin constraint, we can construct a considerably better approximation than that guaranteed by the AGKM algorithm.  Moreover, unlike AGKM, our algorithm does not need to know the parameter $\alpha$.

The proofs of Theorems~\ref{prop:noiseless-case} and~\ref{prop:noisy-case} can be found in the appendix.  The main idea is to show that we can only represent a hott topic efficiently using the hott topic itself.  Some earlier versions of this paper contained incomplete arguments, which we have remedied.  For a signifcantly stronger robustness analysis of Algorithm~\ref{alg:lp-noisycase}, see the recent paper~\cite{Gil12:Robustness-Analysis}.

Having established these theoretical guarantees, it now remains to develop an algorithm to solve the LP.  Off-the-shelf LP solvers may suffice for moderate-size problems, but for large-scale matrix factorization problems, their running time is prohibitive, as we show in Section~\ref{sec:experiments}.  In Section~\ref{sec:igd}, we turn to describe how to solve Algorithm~\ref{alg:lp-noisycase} efficiently for large data sets.

\subsection{Related Work}

Localizing factorizations via column or row subset selection is a popular alternative to direct factorization methods such as the SVD.  
Interpolative decomposition such as Rank-Revealing QR~\cite{Gu96} and CUR~\cite{MahoneyDrineas09} have favorable efficiency properties as compared to factorizations (such as SVD) that are not based on exemplars.  Factorization localization has been used in subspace clustering and has been shown to be robust to outliers~\cite{Elhamifar09,Soltanolkotabi11}.

In recent work on dictionary learning, Esser et al.~and Elhamifar et al.~have proposed a factorization localization solution to nonnegative matrix factorization using group sparsity techniques~\cite{Esser12,Elhamifar12}.  Esser et al.~prove asymptotic exact recovery in a restricted noise model, but this result requires preprocessing to remove duplicate or near-duplicate rows.  Elhamifar shows exact representative recovery in the noiseless setting assuming no hott topics are duplicated.  Our work here improves upon this work in several aspects, enabling finite sample error bounds, the elimination of any need to preprocess the data, and algorithmic implementations that scale to very large data sets.

\section{Incremental Gradient Algorithms for NMF}\label{sec:igd}

The rudiments of our fast implementation rely on two standard optimization techniques: dual decomposition and incremental gradient descent.  Both techniques are described in depth in Chapters 3.4 and 7.8 of Bertsekas and Tstisklis~\cite{BertsekasParallelBook}.

We aim to minimize $\vct{p}^T \diag(\mtx{C})$ subject to $\mtx{C} \in \Phi_{\tau}(\mtx{X})$.  To proceed, form the Lagrangian
\[
	\mathcal{L}(\mtx{C},\beta,\vct{w})=\vct{p}^T \diag(\mtx{C}) + \beta ( \trace(\mtx{C}) - r) + \sum_{i=1}^f w_i \left(\|\mtx{X}_{i\cdot} - [\mtx{C}\mtx{X}]_{i\cdot}\|_1 - \tau\right)
\]
with multipliers $\beta$ and $\vct{w}\geq \mtx{0}$.  Note that we do not dualize out all of the constraints.  The remaining ones appear in the constraint set $\Phi_0 = \{\mtx{C}~:~ \mtx{C}\geq \mtx{0}, \ \diag(\mtx{C})\leq 1, \mbox{ and } C_{ij}\leq C_{jj}~\mbox{for all}~i,j\}$.

Dual subgradient ascent solves this problem by alternating between minimizing the Lagrangian over the constraint set $\Phi_0$, and then taking a subgradient step with respect to the dual variables
\[
	 w_i \leftarrow w_i + s\left(\|\mtx{X}_{i\cdot} - [\mtx{C}^\star\mtx{X}]_{i\cdot}\|_1 - \tau\right) \quad\mbox{and}\quad \beta \leftarrow \beta + s( \trace(\mtx{C}^\star) - r) 
\]
where $\mtx{C}^\star$ is the minimizer of the Lagrangian over $\Phi_0$.  The update of $w_i$ makes very little difference in the solution quality, so we typically only update $\beta$.

We minimize the Lagrangian using projected incremental gradient descent.  Note that we can rewrite the Lagrangian as
\[
\mathcal{L}(\mtx{C},\beta,\vct{w})= -\tau \vct{1}^T\vct{w} - \beta r + \sum_{k=1}^n \left( \sum_{j\in\operatorname{supp}(\mtx{X}_{\cdot k})}  w_j \|\mtx{X}_{jk} - [\mtx{C}\mtx{X}]_{jk}\|_1 +\mu_j (p_j + \beta) C_{jj} \right)\,.
\]
Here, $\operatorname{supp}(\vct{x})$ is the set indexing the entries where $\vct{x}$ is nonzero, and $\mu_j$ is the number of nonzeros in row $j$ divided by $n$.  The incremental gradient method chooses one of the $n$ summands at random and follows its subgradient.  We then project the iterate onto the constraint set $\Phi_0$.  The projection onto $\Phi_0$ can be performed in the time required to sort the individual columns of $\mtx{C}$ plus a linear-time operation.  The full procedure is described in Appendix~\ref{sec:algorithmic-details}.   In the case where we expect a unique solution, we can drop the constraint $C_{ij} \leq C_{jj}$, resulting in a simple clipping procedure: set all negative items to zero and set any diagonal entry exceeding one to one.  In practice, we perform a tradeoff.  Since the constraint $C_{ij} \leq C_{jj}$ is used solely for symmetry breaking, we have found empirically that we only need to project onto $\Phi_0$ every $n$ iterations or so.
	
This incremental iteration is repeated $n$ times in a phase called an \emph{epoch}.  After each epoch, we update the dual variables and quit after we believe we have identified the large elements of the diagonal of $\mtx{C}$.  Just as before, once we have identified the hott rows, we can form $\mtx{W}$ by selecting these rows of $\mtx{X}$.  We can find $\mtx{F}$ just as before, by solving~\eq{cleaning}.  Note that this minimization can also be computed by incremental subgradient descent.  The full procedure, called \name, is described in Algorithm~\ref{alg:sgd-noisycase}.

\begin{algorithm}[t]
  \caption{\name: Approximate Separable NMF by Incremental Gradient Descent}
   \begin{algorithmic}[1]
  \REQUIRE An $f\times n$ nonnegative matrix $\mtx{X}$.  Primal and dual stepsizes $s_p$ and $s_d$.  
  \ENSURE An $f\times r$ matrix $\mtx{F}$ and $r\times n$  matrix $\mtx{W}$ with $\mtx{F}\geq \mtx{0}$, $\mtx{W}\geq \mtx{0}$, and $\ionorm{\mtx{X}-\mtx{F}\mtx{W}}\leq 2\epsilon$.
	\STATE Pick a cost $\vct{p}$ with distinct entries.
	\STATE Initialize $\mtx{C}=\mtx{0}$, $\beta=0$
  \FOR{$t = 1,\ldots, N_{epochs}$}
  \FOR{$i =1,\ldots n$}
  	\STATE Choose $k$ uniformly at random from $[n]$.
	\STATE \label{step:C-update} $\mtx{C}\leftarrow \mtx{C} + s_p \cdot  \operatorname{sign}(\mtx{X}_{\cdot k}-\mtx{C}\mtx{X}_{\cdot k}) \mtx{X}_{\cdot k}^T- s_p \diag( \vct{\mu} \circ (\beta \vct{1} - \vct{p}) )$.
\ENDFOR
	\STATE  Project $\mtx{C}$ onto $\Phi_0$.
	\STATE $\beta \leftarrow \beta + s_d(\trace(\mtx{C})-r)$
\ENDFOR
  \STATE Let $I=\{i~:~C_{ii}=1\}$ and set $\mtx{W}=\mtx{X}_{I\cdot}$.
\STATE Set $\mtx{F} = \arg\min_{\mtx{Z}\in \R^{f\times r}} \ionorm{\mtx{X} - \mtx{Z} \mtx{W}}$
\end{algorithmic} \label{alg:sgd-noisycase}
\end{algorithm}

\subsection{Sparsity and Computational Enhancements for Large Scale.}\label{sec:alg-ls}
For small-scale problems, \name can be implemented in a few lines of Matlab code.  But for the very large data sets studied in Section~\ref{sec:experiments}, we take advantage of natural parallelism and a host of low-level
optimizations that are also enabled by our formulation. As in any numerical program, memory layout and cache
behavior can be critical factors for performance. We use standard
techniques: in-memory clustering to increase prefetching opportunities,
padded data structures for better cache alignment, and compiler
directives to allow the Intel compiler to apply vectorization.

Note that the incremental gradient step (step~\ref{step:C-update} in Algorithm~\ref{alg:sgd-noisycase}) only modifies the entries of $\mtx{C}$ where $\mtx{X}_{\cdot k}$ is nonzero.  Thus, we can parallelize the algorithm with respect to updating either the rows or the columns of $\mtx{C}$.  We store $\mtx{X}$ in large contiguous blocks of memory to encourage hardware prefetching. In contrast, we choose a dense representation of our localizing matrix $\mtx{C}$; this choice trades space for runtime performance. 

Each worker thread is assigned a number of rows of $\mtx{C}$ so that all rows fit in the shared L3 cache. Then,
each worker thread repeatedly scans $\mtx{X}$ while marking updates to multiple rows of $\mtx{C}$. We repeat this process until all rows of $\mtx{C}$ are scanned, similar to the classical block-nested loop join in relational databases~\cite{Shapiro86}.

\section{Experiments}\label{sec:experiments}

Except for the speedup curves, all of the experiments were run on an identical configuration: a dual Xeon X650 (6 cores each) machine with 128GB of RAM.  The kernel is Linux 2.6.32-131.

In small-scale, synthetic experiments, we compared \name to the AGKM algorithm and the linear programming formulation of Algorithm~\ref{alg:lp-noisycase} implemented in Matlab.  Both AGKM and Algorithm~\ref{alg:lp-noisycase} were run using CVX~\cite{cvx} coupled to the SDPT3 solver~\cite{SDPT3}.  We ran \name for $50$ epochs with primal stepsize 1e-1 and dual stepsize 1e-2.  Once the hott topics were identified, we fit $\mtx{F}$ using two cleaning epochs of incremental gradient descent for all three algorithms.

To generate our instances, we sampled $r$ hott topics uniformly from the unit simplex in $\R^n$.  These topics were duplicated $d$ times.  We generated the remaining $f-r(d+1)$ rows to be random convex combinations of the hott topics, with the combinations selected uniformly at random.  We then added noise with $(\infty,1)$-norm error bounded by $\eta \cdot \frac{\alpha^2}{20+13\alpha}$.  Recall that AGKM algorithm is only guaranteed to work for $\eta<1$.  We ran with $f\in\{40,80,160\}$, $n\in\{400,800,1600\}$, $r\in \{3,5,10\}$, $d\in\{0,1,2\}$, and $\eta\in \{0.25, 0.95, 4, 10, 100\}$. Each experiment was repeated $5$ times.

Because we ran over $2000$ experiments with $405$ different parameter settings, it is convenient to use the \emph{performance profiles} to compare the performance of the different algorithms~\cite{Dolan02}. Let $\mathcal{P}$ be the set of experiments and $\mathcal{A}$ denote the set of different algorithms we are comparing.
Let $Q_a(p)$ be the value of some performance metric of the experiment $p\in \mathcal{P}$ for algorithm $a\in \mathcal{A}$.  Then the performance profile at $\tau$ for a particular algorithm is the fraction of the experiments where the value of $Q_a(p)$ lies within a factor of $\tau$ of the minimal value of $\min_{b\in\mathcal{A}} Q_b(p)$.  That is,
\begin{equation*}
P_a(\tau) = \frac{\mathop{\#}\left\{p \in \mathcal{P} ~:~ Q_a(p) \leq \tau \min_{a'\in\mathcal{A}} Q_{a'}(p)\right\}}{\mathop{\#}(\mathcal{P})}\,.
\end{equation*}
In a performance profile, the higher a curve corresponding to an algorithm, the more often it outperforms the other algorithms.  This gives a convenient way to  contrast algorithms visually.

\begin{figure}[t]
\centering
\begin{tabular}{ccc}
\includegraphics[width=.3\textwidth]{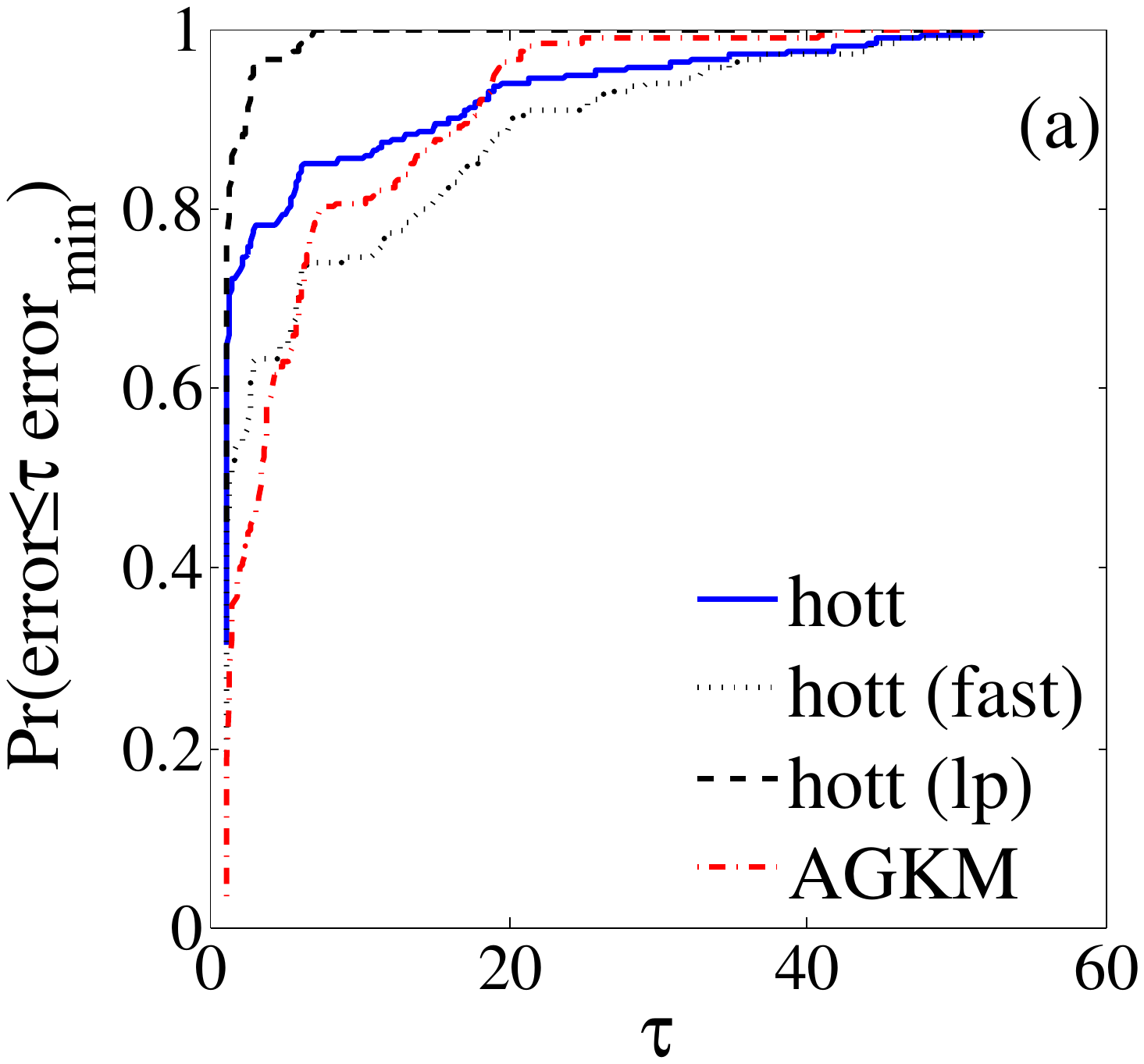} & 
\includegraphics[width=.3\textwidth]{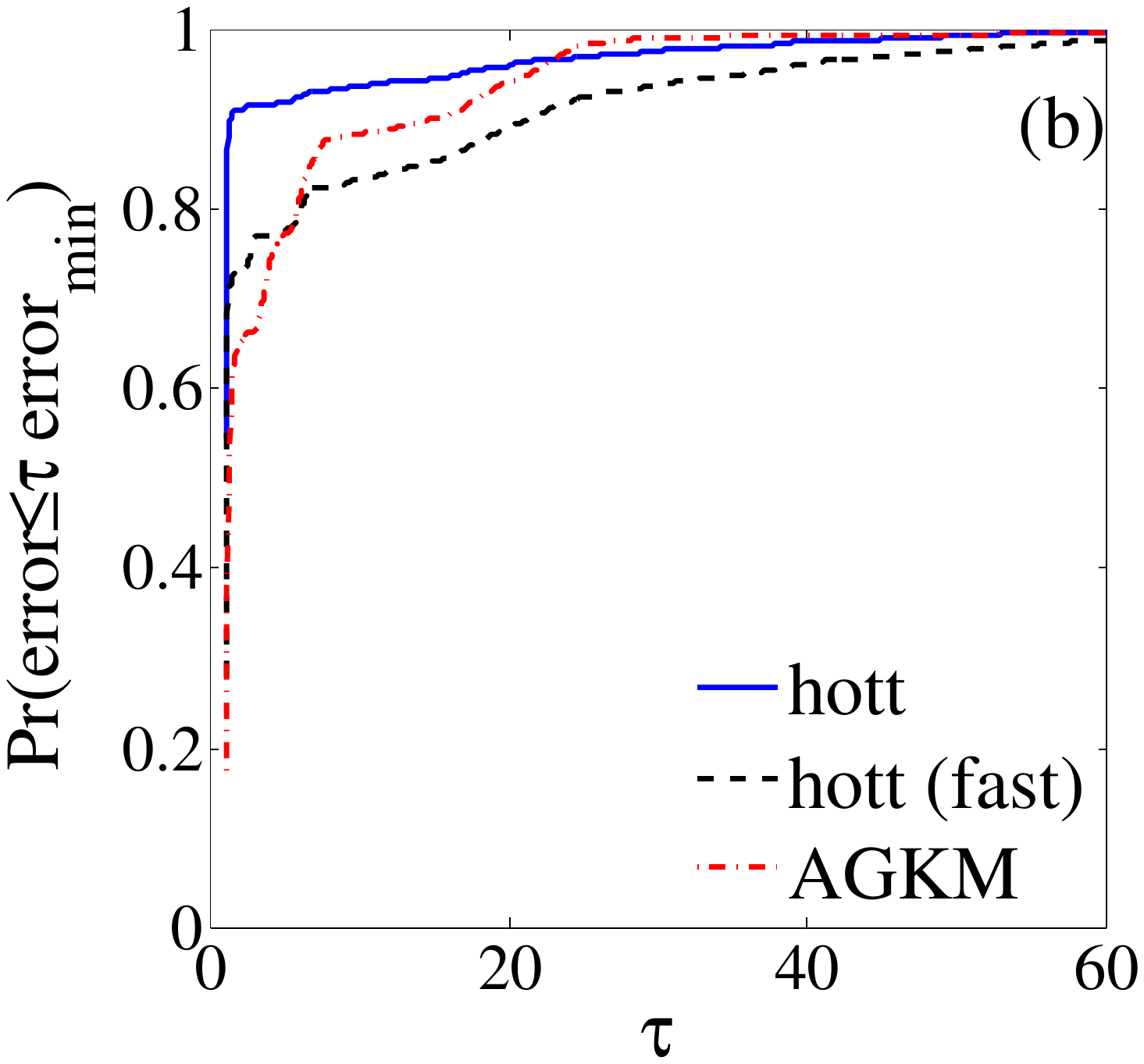} & 
\includegraphics[width=.3\textwidth]{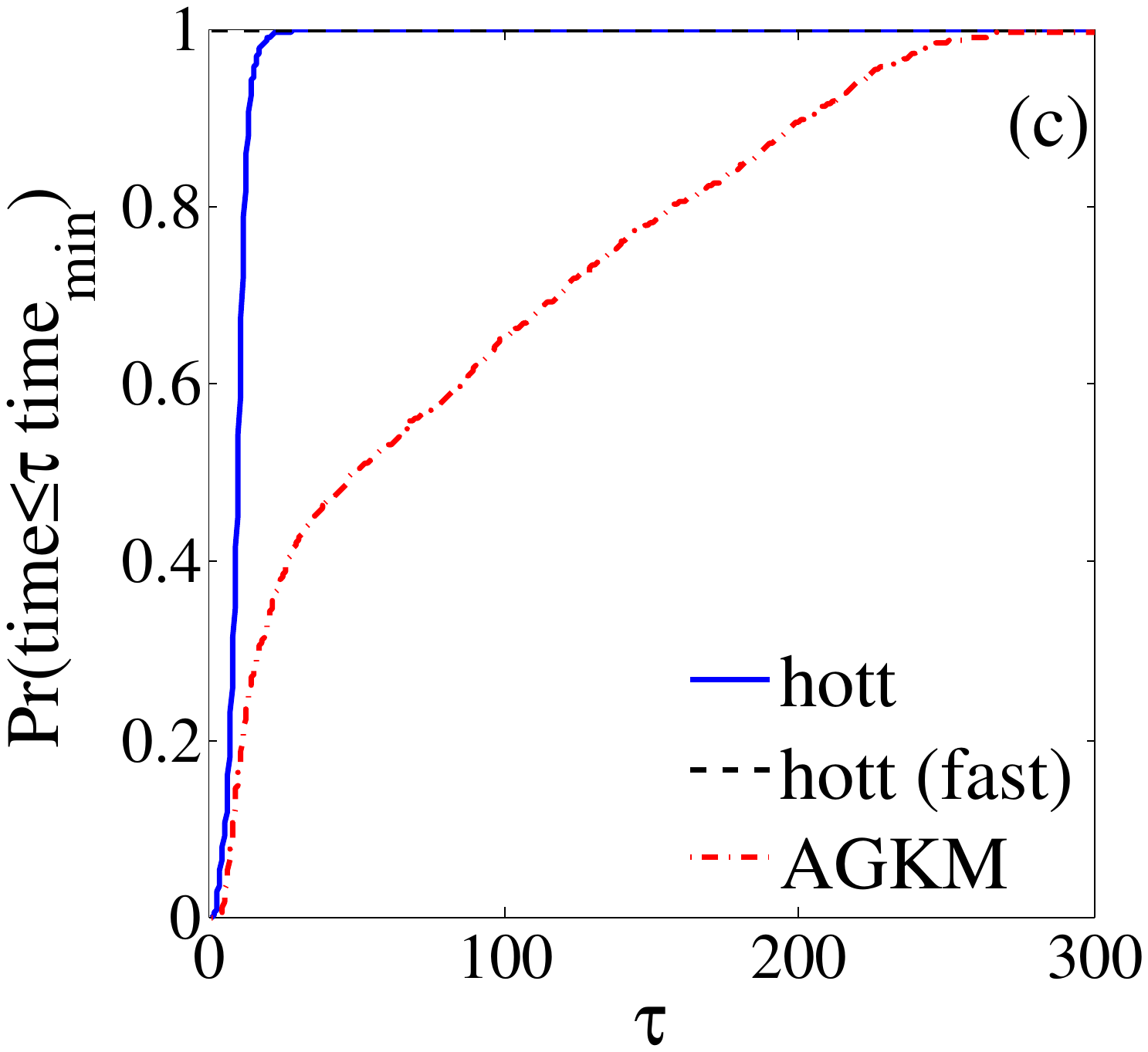}\\
\includegraphics[width=.3\textwidth]{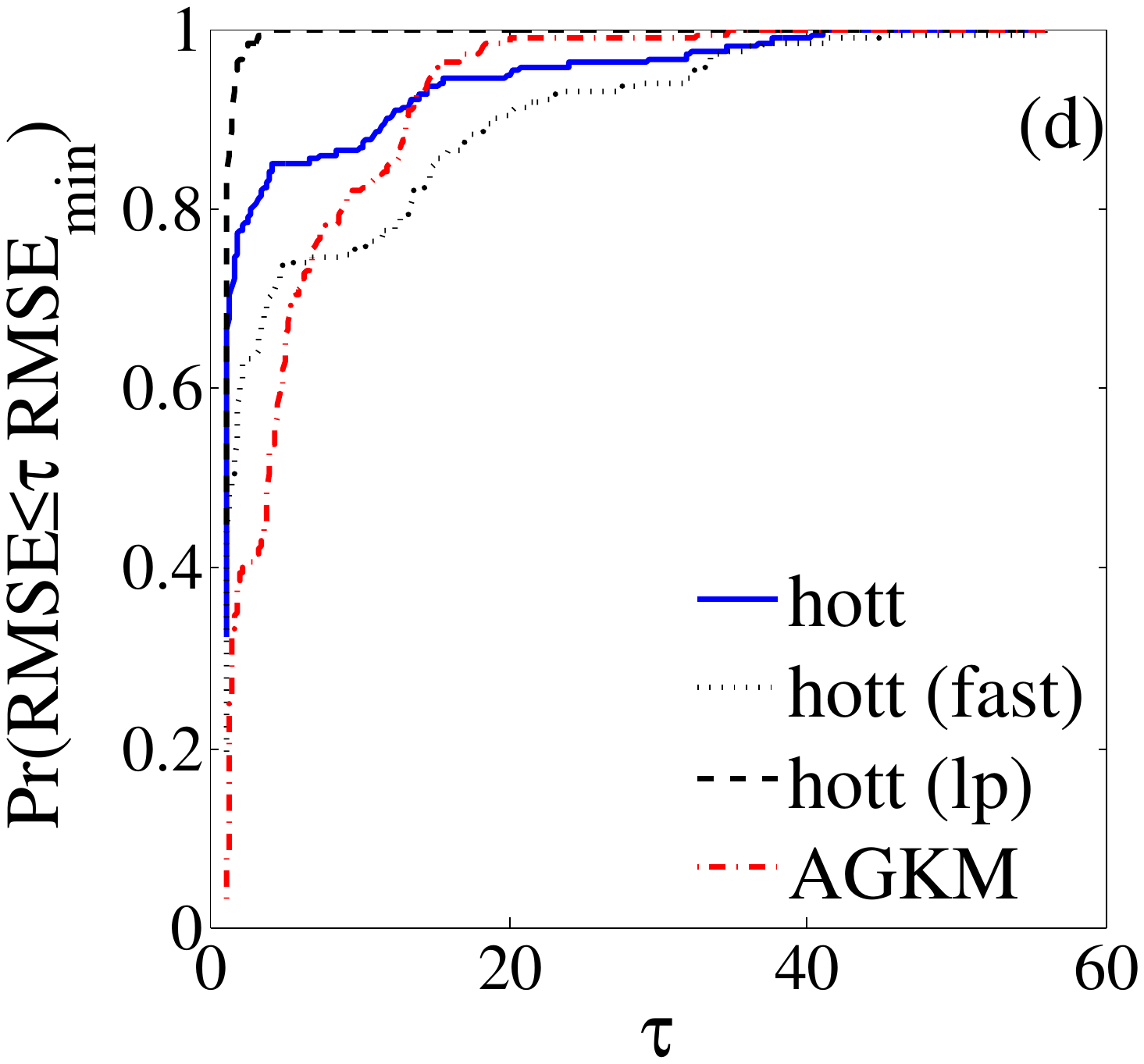} & 
\includegraphics[width=.3\textwidth]{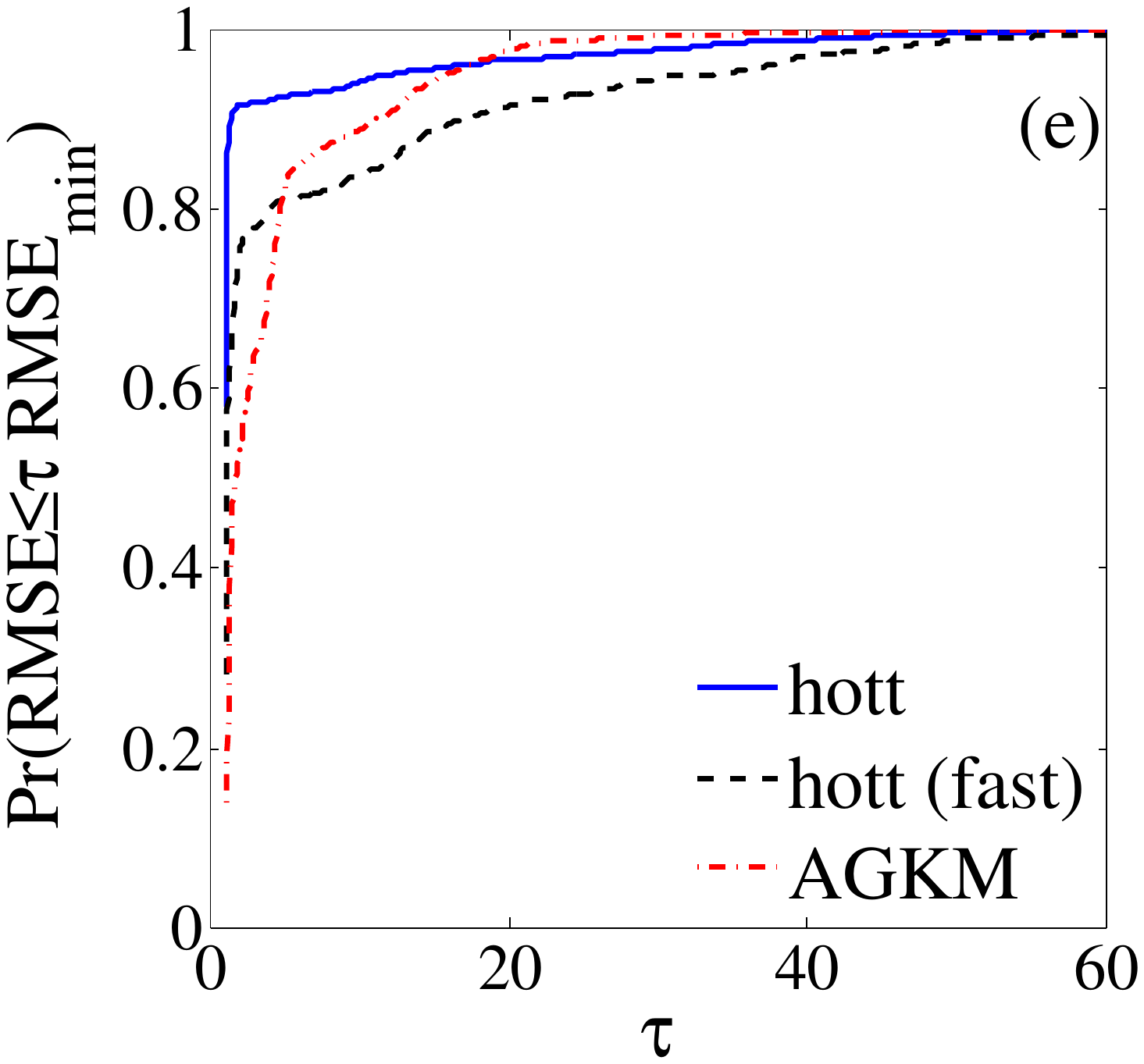} &
\includegraphics[width=.3\textwidth]{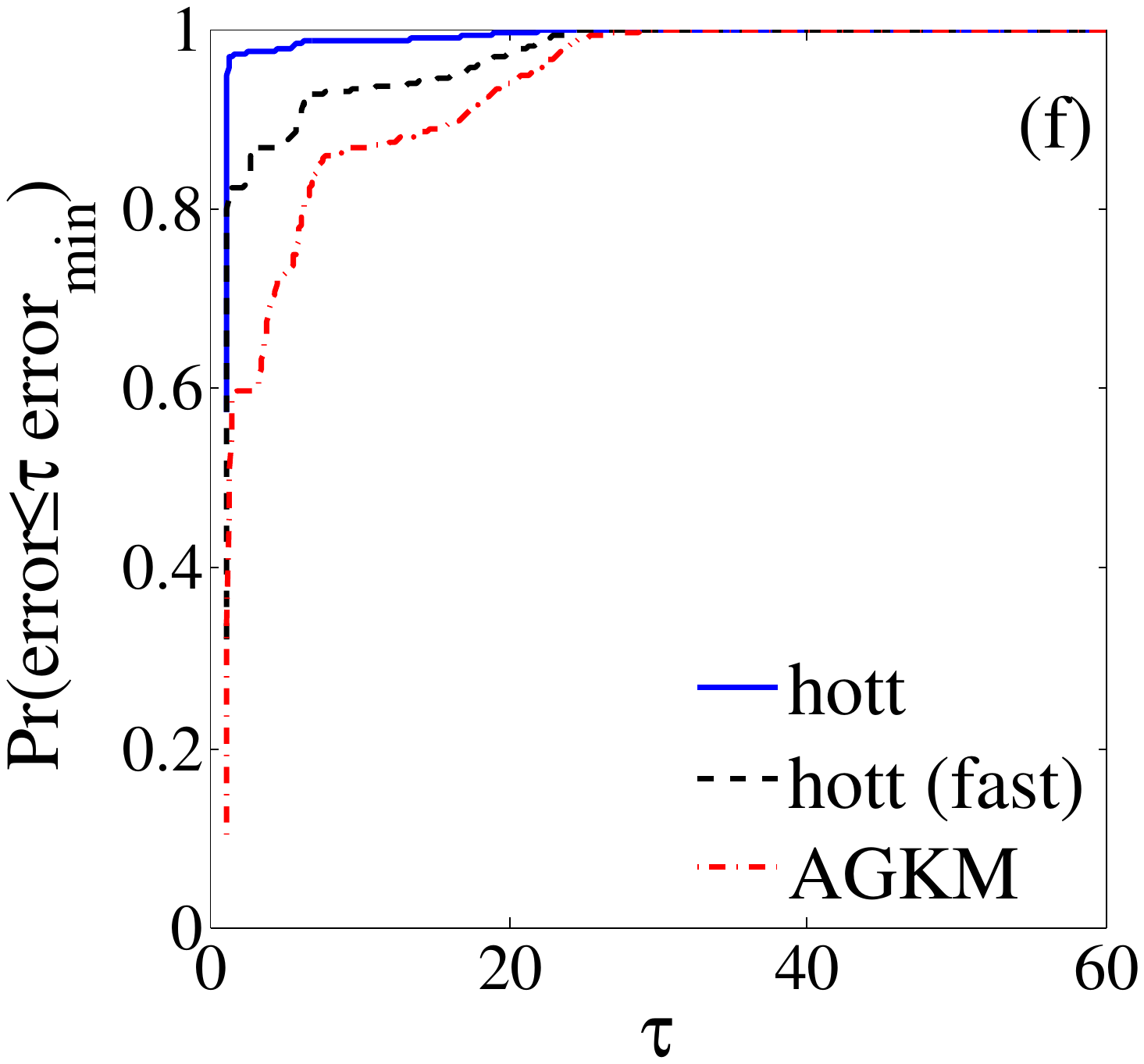}  \\
\end{tabular}
\caption{\small Performance profiles for synthetic data.  (a) $(\infty,1)$-norm error for  $40\times 400$ sized instances and (b) all instances. (c) is the performance profile for running time on all instances. RMSE performance profiles for the (d) small scale and (e) medium scale experiments.   (f) $(\infty,1)$-norm error for the $\eta\geq 1$.  In the noisy examples, even 4 epochs of \name is sufficient to obtain competitive reconstruction error.
\label{fig:perf-profiles-rmse}
\label{fig:perf-profiles}}
\end{figure}

Our performance profiles are shown in Figure~\ref{fig:perf-profiles}.  The first two figures correspond to experiments with $f=40$ and $n=400$.  The third figure is for the synthetic experiments with all other values of $f$ and $n$.  In terms of $(\infty,1)$-norm error, the linear programming solver typically achieves the lowest error.  However, using SDPT3, it is prohibitively slow to factor larger matrices.  On the other hand, \name achieves better noise performance than the AGKM algorithm in much less time.  Moreover, the AGKM algorithm must be fed the values of $\epsilon$ and $\alpha$ in order to run.  \name does not require this information and still achieves about the same error performance.

We also display a graph for running only four epochs (hott (fast)).  This algorithm is by far the fastest algorithm, but does not achieve as optimal a noise performance.  For very high levels of noise, however, it achieves a lower reconstruction error than the AGKM algorithm, whose performance degrades once $\eta$ approaches or exceeds $1$ (Figure~\ref{fig:perf-profiles-rmse}(f)).   We also provide performance profiles for the root-mean-square error of the nonnegative matrix factorizations (Figure~\ref{fig:perf-profiles-rmse} (d) and (e)).  The performance is qualitatively similar to that for the $(\infty,1)$-norm.

We also coded \name in C++, using the design principles described in Section~\ref{sec:alg-ls}, and ran on three large data sets.  We generated a large synthetic example (jumbo) as above with $r=100$.  We generated a co-occurrence matrix of people and places from the ClueWeb09 Dataset~\cite{clueweb}, normalized by TFIDF.  We also used \name to select features from the RCV1 data set to recognize the class CCAT~\cite{Lewis04}. The statistics for these data sets can be found in Table~\ref{fig:table-compare}. 

\begin{table}
\small
\begin{center}
\begin{tabular}{|c|rrr|r|r|}
\hline
data set & features & documents & nonzeros & size (GB) & time (s) \\
\hline
jumbo   & 1600 & 64000 & 1.02e8 & 2.7& 338\\
clueweb   & 44739 & 351849 & 1.94e7 & 0.27 & 478\\
RCV1 & 47153 & 781265 & 5.92e7 & 1.14 & 430\\
\hline
\end{tabular}
\caption{\small Description of the large data sets. Time is to find 100 hott topics on the 12 core machines.}\label{fig:table-compare}
\end{center}
\end{table}

In Figure~\ref{fig:speedup} (left), we plot the speed-up over a serial implementation. In contrast to other parallel methods that exhibit memory contention~\cite{Hogwild}, we see superlinear speed-ups for up to 20 threads due to hardware prefetching and cache effects.  All three of our large data sets can be trained in minutes, showing that we can scale \name on both synthetic and real data.  Our algorithm is able to correctly identify the hott topics on the jumbo set. For clueweb, we plot the RMSE Figure~\ref{fig:speedup} (middle).   This curve rolls off quickly for the first few hundred topics, demonstrating that our algorithm may be useful for dimensionality reduction in Natural Language Processing applications. For RCV1, we trained an SVM on the set of features extracted by \name and  plot the misclassification error versus the number of topics in Figure~\ref{fig:speedup} (right). With $1500$ hott topics, we achieve $7\%$ misclassification error as compared to $5.5\%$ with the entire set of features.

\begin{figure}
\centering
\begin{tabular}{ccc}
\includegraphics[width=.3\textwidth]{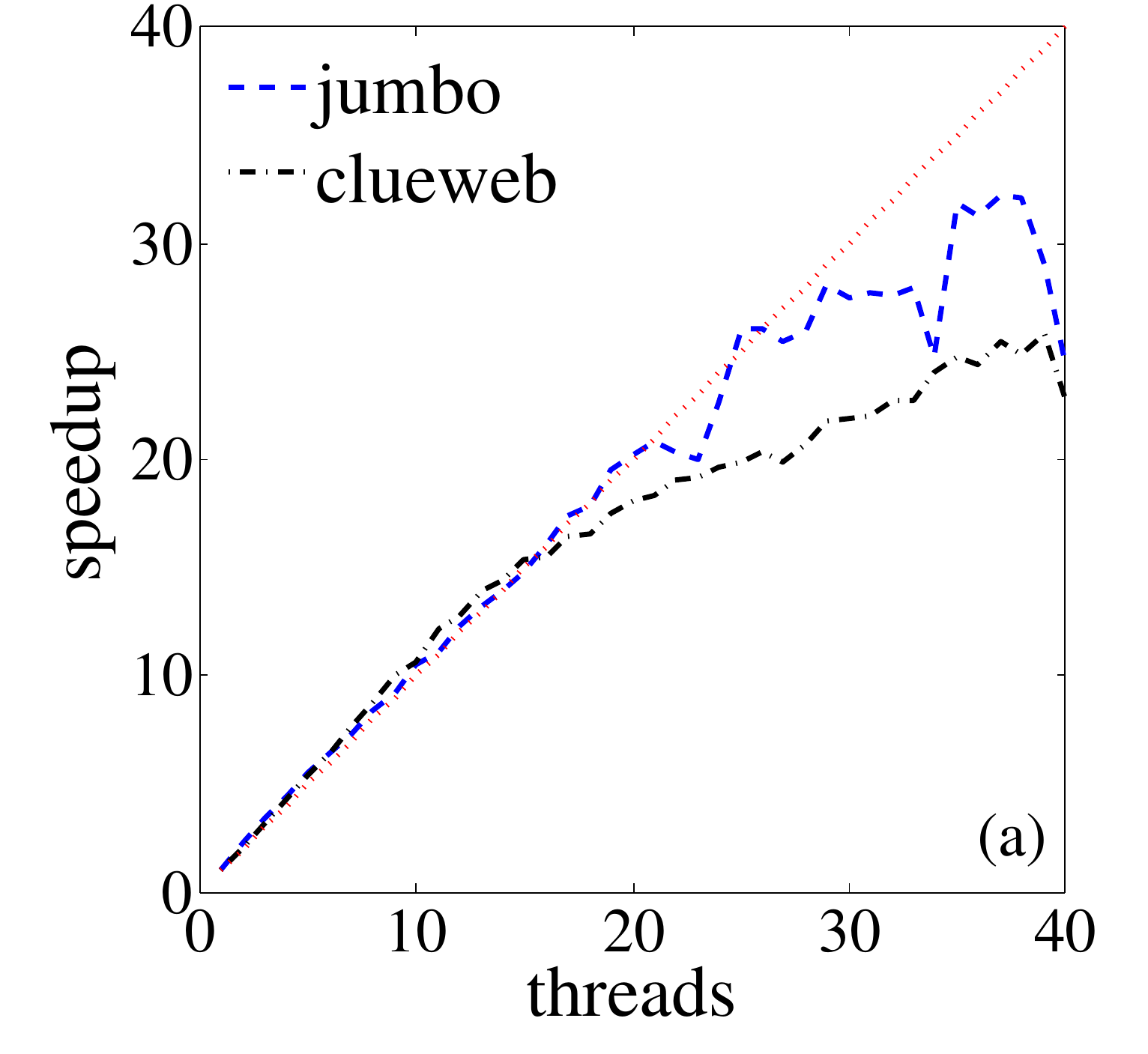} & 
\includegraphics[width=.3\textwidth]{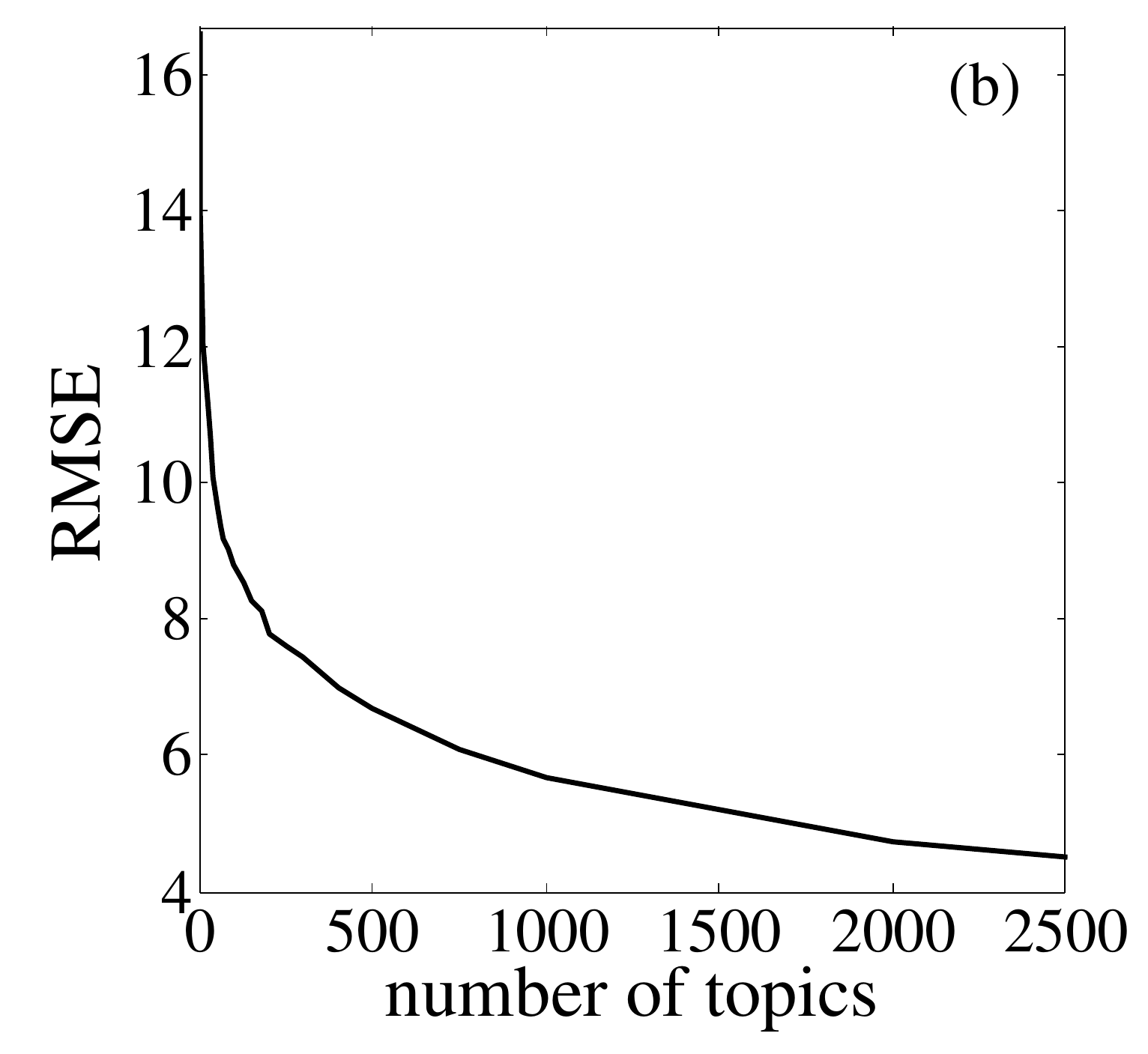} & 
\includegraphics[width=.3\textwidth]{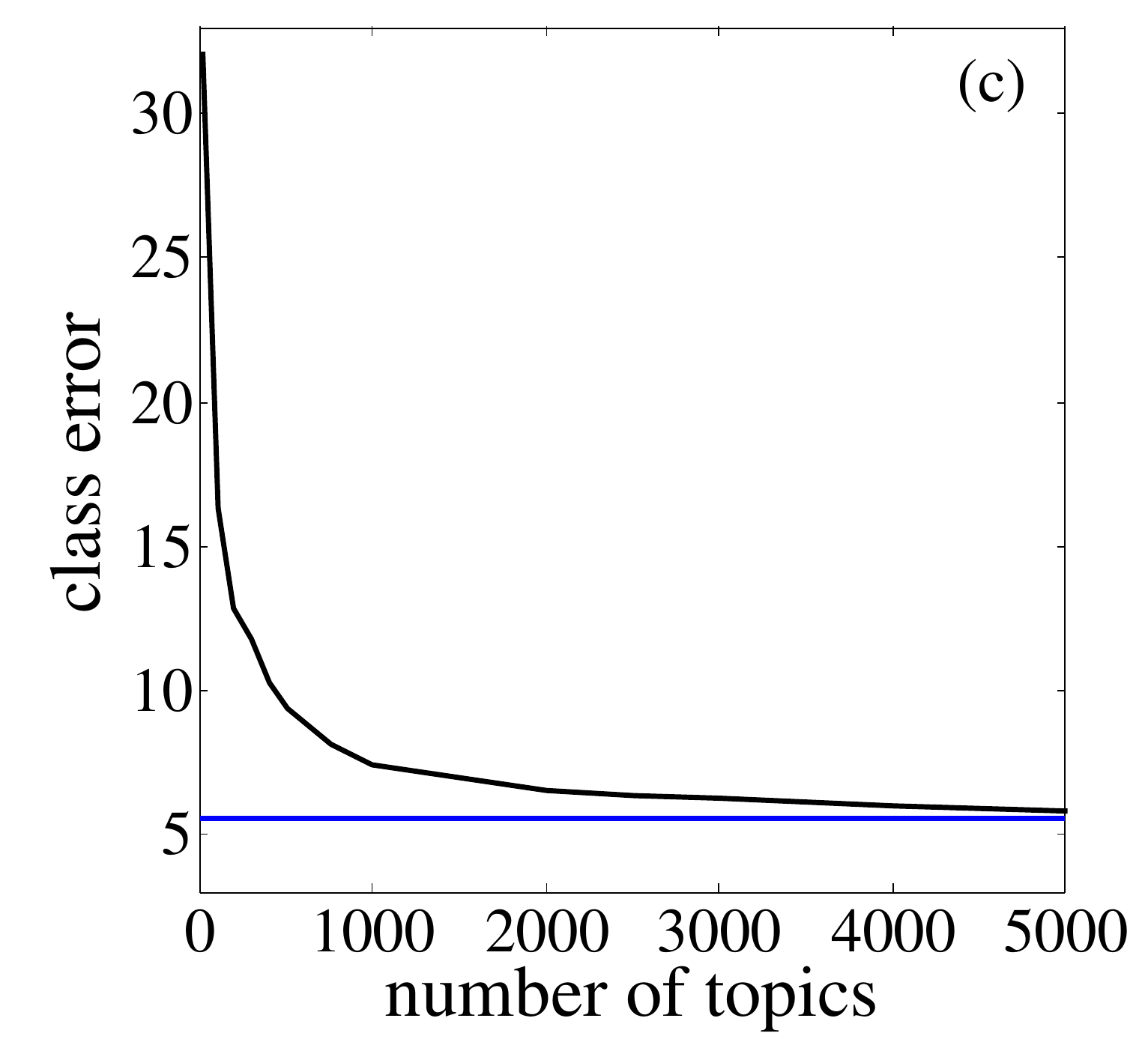}
\end{tabular}
\caption{\small (left) The speedup over a serial implementation for \name on the jumbo and clueweb data sets.  Note the superlinear speedup for up to 20 threads.  (middle)  The RMSE for the clueweb data set. (right) The test error on RCV1 CCAT class versus the number of hott topics.  The horizontal line indicates the test error achieved using all of the features.
\label{fig:speedup}}
\end{figure}

\section{Discussion}

This paper provides an algorithmic and theoretical framework for analyzing and deploying any factorization problem that can be posed as a linear (or convex) factorization localizing program.  Future work should investigate the applicability of \name to other factorization localizing algorithms, such as subspace clustering, and should revisit earlier theoretical bounds on such prior art.

\subsubsection*{Acknowledgments}
The authors would like to thank Sanjeev Arora, Michael Ferris, Rong Ge, Nicolas Gillis, Ankur Moitra, and Stephen Wright for helpful suggestions.  BR is generously supported by ONR award N00014-11-1-0723, NSF award CCF-1139953, and a Sloan Research Fellowship.  CR is  generously supported by NSF CAREER award under IIS-1054009, ONR award N000141210041, and gifts or research awards from American Family Insurance, Google, Greenplum, and Oracle. JAT is generously supported by ONR award N00014-11-1002, AFOSR award FA9550-09-1-0643, and a Sloan Research Fellowship.
{\small
\bibliographystyle{abbrv}
\bibliography{hott}
}

\newpage

\appendix

\section{Proofs}

Let $\mtx{Y}$ be a nonnegative matrix whose rows sum to one.
Assume that $\mtx{Y}$ admits an exact separable factorization
of rank $r$.  In other words, we can write $\mtx{Y} = \mtx{FW}$ where the rows
of $\mtx{W}$ are $\alpha$-robust simplicial and 
$$
\mtx{\Pi} \bm{F} = \begin{bmatrix} \mathbf{I}_r \\ \mtx{M} \end{bmatrix}
$$
for some permutation $\mtx{\Pi}$.
Let $I$ denote the indices of the rows in $\mtx{Y}$ that correspond with the
identity matrix in the factorization, which we have called the \emph{hott rows}.
Then we can write each row $j$ that is not hott as a convex combination
of the hott rows:
$$
\bm{Y}_{j\cdot} = \sum_{k\in I} M_{jk} \bm{Y}_{k\cdot}
\quad\text{for each $j \notin I$.}
$$
As we have discussed, we may assume that $\sum_{k} M_{jk} = 1$
for each $j \notin I$ because each row of $\mtx{Y}$ sums to one.

The first lemma offers a stronger bound on the coefficients
$M_{jk}$ in terms of the distance between row $j$ and the
hott rows.

\begin{lemma}\label{lemma:M-bound}
For an index $\ell$, suppose that the row $\mtx{Y}_{\ell\cdot}$ has distance greater than $\delta$ from a hott topic $\mtx{Y}_{i \cdot}$ with $i \in I$.  Then  $M_{\ell i} \leq 1-\delta/2$.
\end{lemma}

\begin{proof}
We can express the $\ell$th row as a convex combination of hott rows: $\mtx{Y}_{\ell \cdot} = \sum_{k\in I} M_{\ell k} \mtx{Y}_{k \cdot}$.  For each $i \in I$, we can bound $M_{\ell i}$ as follows.
\[
\begin{aligned}
	\delta \leq \|\mtx{Y}_{i \cdot} - \mtx{Y}_{\ell \cdot}\|_1 &= \left\| \mtx{Y}_{i \cdot} -  \sum_{k\in I} M_{\ell k} \mtx{Y}_{k \cdot} \right\|_1\\
	&= \left\| (1-M_{\ell i}) \mtx{Y}_i -  \sum_{k\in I \setminus \{i\} } M_{\ell k} \mtx{Y}_{k \cdot} \right\|_1\\
	&\leq \left\|(1-M_{\ell i})   \mtx{Y}_i \right\|_1 + \sum_{k\in I \setminus \{i\} } M_{\ell k} \left\| \mtx{Y}_{k \cdot} \right\|_1 \\
	&= 2 (1-M_{\ell i})\,.
\end{aligned}
\]
The inequality is the triangle inequality.  To reach the last line, we use the fact that each row of $\mtx{Y}$ has $\ell_1$ norm equal to one.  Furthermore, $\sum_{k \in I \setminus \{i\}} M_{\ell k} = 1 - M_{\ell i} \geq 0$ because each row of $\bm{M}$ consists of nonnegative numbers that sum to one.
Rearrange to complete the argument.
\end{proof}

The next lemma is the central tool in our proofs.  It tells us that any representation of a hott row has to involve rows that are close in $\ell_1$ norm to a hott row.  To state the result, we define for each hott row $i$
$$
\mathcal{B}_{\delta}(i) = \{ j : \Vert \vct{Y}_{i\cdot} - \mtx{Y}_{j\cdot} \Vert_1 \leq \delta \}.
$$
In other words, $\mathcal{B}_{\delta}(i)$ contains the indices of all rows with
$\ell_1$ distance no greater than $\delta$ from the hott topic $\mtx{Y}_{i\cdot}$.

\begin{lemma}\label{lemma:weight-bound}
Let $\vct{c} \in \R^f$ be a nonnegative vector whose entries sum to one.  For some hott row $i \in I$, suppose that $\|\vct{c}^T \mtx{Y} -\mtx{Y}_{i\cdot} \|_{1} \leq \tau$. Then
\begin{align}
\label{eq:weight-bound-a}	\sum_{j \in \mathcal{B}_\delta (i) } c_{j}
	&\geq 1 - \frac{2\tau}{\min\{\alpha \delta, \alpha^2 \}}.
\end{align}
\end{lemma}

\begin{proof}
Let us introduce notation for the quantity of interest: $w_i = w_i(\vct{c}) = \sum_{j \in \mathcal{B}_\delta (\mtx{X}_{i\cdot}) } c_{j}$.  We may assume that $w_i < 1$, or else the result holds trivially.
Since the entries of $\vct{c}$ sum to one, we have
$$
0 < 1 - w_i = \sum_j c_j - \sum_{j \in \mathcal{B}_{\delta}(i)} c_j
	= \sum_{j \notin \mathcal{B}_\delta (i) } c_{j}.
$$

Next, we introduce the extra assumption that $\delta < \alpha$.  It is clear that $w_i$ increases monotonically with $\delta$, so any lower bound on $w_i$ that we establish in this case extends to a bound that holds for larger $\delta$.  Since the hott topics are $\alpha$-robust simplicial, all the other hott topics are at least $\delta$ away from $\mtx{Y}_{i\cdot}$ in $\ell_1$ norm.  Therefore, the hott row $i$ is the unique hott row listed in $\mathcal{B}_{\delta}(i)$.

To establish the result, we may as well assume that $w_i(\vct{c})$ achieves its minimum possible value subject to the constraints that the value of $\vct{c}^T \mtx{Y}$ is fixed and that $\vct{c}$ is a nonnegative vector that sums to one.  We claim that this minimum such $w_i$ occurs if and only if $c_j = 0$ for all $j \in \mathcal{B}_{\delta}(i) \setminus \{i\}$.  We complete the proof under this additional surmise.

The assertion in the last paragraph follows from an argument by contradiction.  Suppose that $w_i(\vct{c})$ were minimized at a vector $\vct{c}$ where $c_{j} > 0$ for some $j \in \mathcal{B}_{\delta}(i) \setminus \{i\}$.  Then we can construct another set of coefficients $\tilde{\vct{c}}$ that satisfies the constraints and leads to a smaller value of $w_i$.  We have the representation $\mtx{Y}_{j\cdot} = \sum_{k \in I} M_{jk} \mtx{Y}_{k \cdot}$.  Set $\tilde{c}_{j} = 0$; set $\tilde{c}_k = c_k + c_{j} M_{jk}$ for each $k \in I$, and set $\tilde{c}_k = c_k$ for all remaining $k \notin I \cup \{j\}$.  It is easy to verify that $\tilde{\vct{c}}^T \mtx{Y} = \vct{c}^T \mtx{Y}$ and that $\tilde{\vct{c}}$ is a nonnegative vector whose entries sum to one.  But the value of $w_i$ is strictly smaller with the coefficients $\tilde{\vct{c}}$:
$$
w_i(\tilde{\vct{c}}) =
\sum_{k \in \mathcal{B}_{\delta} (i)} \tilde{c}_k < 
\sum_{k \in \mathcal{B}_{\delta} (i)} c_k
= w_i(\vct{c})
$$
In this relation, all the summands cancel, except for the one with index $j$.  But $\tilde{c}_j = 0 < c_j$.  It follows that the minimum value of $w_i$ cannot occur when $c_j > 0$.  Compactness of the constraint set assures us that there is some vector $\vct{c}$ of coefficients that minimizes $w_i(\vct{c})$, so we must conclude that the minimizer $\vct{c}$ satisfies $c_j = 0$ for $j \in \mathcal{B}_{\delta}(i) \setminus \{i\}$.

Let us continue. Owing to the assumption that $\mtx{Y}_{i\cdot}$ is no farther than $\tau$ from $\vct{c}^T \mtx{Y}$, we have
\begin{align}
\nonumber	\tau
	&\geq \left\| \mtx{Y}_{i\cdot} - \sum_{j} c_j \mtx{Y}_{j \cdot} \right\|_1 \\ \nonumber
	&= \left\|(1-w_i) \mtx{Y}_{i\cdot} - \sum_{j \notin \mathcal{B}_\delta (i) } c_j \mtx{Y}_{j \cdot} \right\|_1\\
\label{eq:tau-lower-bound}	& = (1-w_i) \left\| \mtx{Y}_{i\cdot} - \frac{1}{1-w_i}\sum_{j \notin \mathcal{B}_\delta (i) } \sum_{k \in I} c_j M_{jk} \mtx{Y}_{k \cdot} \right\|_1 \,.
\end{align}
The first line follows when we split the sum over $j$ based on whether or not the components fall in $\mathcal{B}_\delta(i)$.  Then we apply the property that $c_j = 0$ for $j \in \mathcal{B}_{\delta}(i) \setminus \{i\}$, and we identify the quantity $w_i$.  In the last line, we factored out $1 - w_i$, and we introduced the separable factorization of $\mtx{Y}$.

Next, for each $k \in I$, define
\[
	\pi_k :=  \frac{1}{1-w_i} \sum_{j \notin \mathcal{B}_\delta (i) } c_j M_{jk}\,,
\]
and note that $\pi_k\geq 0$.
Furthermore,
$$
\sum_{k \in I} \pi_k = \frac{1}{1-w_i} \sum_{j \notin \mathcal{B}_{\delta}(i)} c_{j} \sum_{k\in I} M_{jk}
	= \frac{1}{1-w_i} \sum_{j\notin \mathcal{B}_{\delta}(i)} c_{j}
	= 1
$$
because the rows of $\mtx{M}$ sum to one and because of the definition of $w_i$.  Lemma~\ref{lemma:M-bound} implies that $\pi_i$ satisfies the bound
\begin{equation} \label{eqn:pi-bound}
\pi_i = \frac{1}{1-w_i} \sum_{j \notin \mathcal{B}_{\delta}(i)} c_{j} M_{ji}
	\leq \frac{1 - \delta/2}{1 - w_i} \sum_{j \notin \mathcal{B}_{\delta}(i)} c_{j}
	= 1 - \delta/2.
\end{equation}
Indeed, the lemma is valid because $\mtx{Y}_{j\cdot}$ is at least a distance of $\delta$ away from $\bm{Y}_{i\cdot}$ for every $j \notin \mathcal{B}_{\delta}(i)$.

With these observations, we can continue our calculation from~\eqref{eq:tau-lower-bound}:
\begin{align*}
\tau &\geq (1 - w_i) \left\Vert \mtx{Y}_{i \cdot} - \sum_{k \in I} \pi_k \mtx{Y}_{k\cdot} \right\Vert_1 \\
	&= (1 - w_i) (1-\pi_i) \left\Vert \mtx{Y}_{i\cdot} - \sum_{k \in I \setminus\{i\}} \frac{\pi_k}{1-\pi_i} \mtx{Y}_{k\cdot} \right\Vert_1 \\
	&\geq (1 - w_i) (1-\pi_i) \alpha \\
	&\geq (1 - w_i) (\delta / 2) \alpha.
\end{align*}
The first identity follows when we combine the $i$th term in the sum with $\mtx{Y}_{i\cdot}$.  The inequality depends on the assumption that $\mtx{W}$ is $\alpha$-robust simplicial; any
convex combination of $\{\mtx{Y}_{k\cdot} : k \notin I\}$ is at least $\alpha$ away from $\mtx{Y}_{i\cdot}$
in $\ell_1$ norm.  Afterward, we use the bound~\eqref{eqn:pi-bound}.  Rearrange the final expression to complete the argument.
\end{proof}

\subsection{Proof of Theorem~\ref{prop:noiseless-case}}

This result is almost obvious when there are no duplicated rows.  Indeed, since the hott topics form a simplicial set and the matrix $\mtx{Y}$ admits a separable factorization, the only way we can represent all $r$ hott topics exactly is to have $C_{ii} = 1$ for every hott row $i$.  This exhausts the trace constraint, and we see that every other diagonal entry $C_{kk} = 0$ for every not hott row $k$.  The only matrices that are feasible identify the hott rows on the diagonal.  They must represent the remaining rows using linear combinations of the hott topics because of the constraints $\mtx{CY} = \mtx{Y}$ and $C_{ij} \leq C_{jj}$.  It follows that the only feasible matrices are factorization localizing matrices.

When there are duplicated rows, the analysis is slightly more delicate.  By the same argument as above, all the weight on the diagonal must be concentrated on hott rows.  But the objective $\vct{p}^T \diag(\mtx{C})$ ensures that, out of any set of duplicates of a given topic, we always pick the duplicate row $j$ where $p_j$ is smallest; otherwise, we could reduce the objective further.  Therefore, the diagonal of $\mtx{C}$ identifies all $r$ distinct hott topics, and we select each one duplicate of each topic.  As before, the other constraints ensure that the remaining rows are represented with this distinguished choice of hott topic exemplars.  Therefore, the only minimizers are factorization localizing matrices that identify each hott topic exactly once.

\subsection{Proof of Theorem~\ref{prop:noisy-case}}

Let $\mtx{X} = \mtx{Y} + \mtx{\Delta}$.  The matrix $\mtx{X}$ is the observed data, with rows scaled to have unit sum, and the perturbation matrix $\mtx{\Delta}$ satisfies $\ionorm{\mtx{\Delta}} \leq \epsilon$.  We assume that $\mtx{Y}$ is a nonnegative matrix whose rows sum to one, and we posit that it admits a rank-$r$ separable NMF $\mtx{Y} = \mtx{FW}$ where $\mtx{W}$ is $\alpha$-robust simplicial.  We write $I$ for the set of rows corresponding to hott topics in $\mtx{Y}$.

Suppose that $\mtx{C}_0$ is a factorization localizing matrix for the underlying matrix $\mtx{Y}$.  That is, $\mtx{C}_0 \mtx{Y} = \mtx{Y}$ and each row of $\mtx{C}_0$ sums to one.  It follows that
$$
\ionorm{ \mtx{C}_0 \mtx{\Delta} - \mtx{\Delta} } \leq (\ionorm{\mtx{C}_0} + \ionorm{\mathbf{I}}) \ionorm{ \mtx{\Delta} }
	\leq 2\epsilon.
$$
Using our decomposition $\mtx{X} = \mtx{Y} + \mtx{\Delta}$, we quickly verify that
\[
\ionorm{ \mtx{C}_0 \mtx{X} - \mtx{X}}
	\leq \ionorm{ \mtx{C}_0 \mtx{Y} - \mtx{Y} } + \ionorm{\mtx{C}_0 \mtx{\Delta} - \mtx{\Delta}}
	\leq 2 \epsilon.
\]
The point here is that a factorization localizing matrix for $\mtx{Y}$ serves as an approximate factorization localizing matrix for $\mtx{X}$.

Our approach for constructing an approximate factorization of $\mtx{X}$ requires us to minimize a cost function $\vct{t}^T \diag(\mtx{C})$ over the constraint set
\begin{equation}\label{eq:robust-feasibility-dup}
\Phi_{2\epsilon}(\mtx{X}) = \left\{ \mtx{C} \geq \mtx{0} :
	\ionorm{ \mtx{CX} - \mtx{X} } \leq 2\epsilon, \trace(\mtx{C}) = r, C_{jj} \leq 1 \ \forall j,~C_{ij}\leq C_{jj}\ \forall i, j \right\}.
\end{equation}
Note that the factorization localizing matrix $\mtx{C}_0$ for $\mtx{Y}$ is a member of this set, so the optimization problem we solve in Theorem~\ref{prop:noisy-case} is feasible.

Suppose that $\mtx{C} \in \Phi_{2\epsilon}(\mtx{X})$ is arbitrary.  Let us check that the row sums of $\mtx{C}$ are not much larger than one.  To that end, note that
$$
\mtx{C}\vct{1} = \mtx{CX}\vct{1} = \mtx{X} \vct{1} + (\mtx{CX} - \mtx{X}) \vct{1}
	= \vct{1} + (\mtx{CX} - \mtx{X}) \vct{1}.
$$
We have twice used the fact that every row of $\mtx{X}$ sums to one.
For any row $\vct{c}$ of the matrix $\mtx{C}$, this formula yields $\vct{c}^T \vct{1} \leq 1 + 2\epsilon$ since $\ionorm{\mtx{CX} - \mtx{X}} \leq 2\epsilon$.   As a consequence,
$$
\ionorm{ \mtx{C\Delta} - \mtx{\Delta} } \leq (\ionorm{ \mtx{C} } + \ionorm{\mathbf{I}} )
	\ionorm{\mtx{\Delta}} \leq (1 + 2\epsilon + 1) \epsilon = 2\epsilon + 2\epsilon^2.
$$
We may conclude that
\[
\ionorm{ \mtx{C} \mtx{Y} - \mtx{Y}}
	\leq \ionorm{ \mtx{CX} - \mtx{X} } + \ionorm{\mtx{C} \mtx{\Delta} - \mtx{\Delta}}
	\leq 4 \epsilon + 2\epsilon^2.
\]

The margin assumption states that $\| \vct{Y}_{\ell \cdot} - \vct{Y}_{i \cdot} \| > d_0$ for every hott topic $i \in I$ and every row $\ell \notin I$.  For any $i \in I$, Lemma~\ref{lemma:weight-bound} ensures that any approximate representation $\vct{c}^T \mtx{Y}$ of the $i$th row $\mtx{Y}_{i\cdot}$ with error at most $4\epsilon + 2\epsilon^2$ satisfies
$$
c_i = \sum_{j \in \mathcal{B}_{d_0}(i)} c_j \geq 1 - \frac{8\epsilon + 4 \epsilon^2}{\min\{\alpha d_0, \alpha^2\}}.
$$
In particular, every matrix $\mtx{C}$ in the set $\Phi_{2\epsilon}(\mtx{X})$ has $C_{ii} \geq 1 - (8\epsilon + 4\epsilon^2)/\min\{\alpha d_0, \alpha^2\}$ for each hott topic $i$.  To ensure that hott topic $i$ has weight $C_{ii}$ greater than $1 - 1/(r+1)$ for each $i$, we need
$$
\epsilon < \sqrt{1 + \frac{\min\{\alpha d_0, \alpha^2\}}{4(r+1)}} - 1
	< \frac{\min\{\alpha d_0, \alpha^2\}}{9(r+1)}
$$
Since there are $r$ hott rows, they carry total weight greater than $r(1 - 1/(r+1))$.  Given the trace constraint, that leaves less than $1 - 1/(r+1)$ for the remaining rows.  We see that each of the $r$ hott rows must carry more weight than every row that is not hott, so we can easily identify them.

Once we have identified the set $I$ of hott topics, we simply solve the second linear program
\begin{equation}\label{eq:debiasing}
	\minimize_{\mtx{B}} \ionorm{\mtx{X} -  \left[\begin{array}{c} \mathbf{I}\\ \mtx{B} \end{array} 
	\right]\mtx{X}_I}
\end{equation}
to find a $2\epsilon$-accurate factorization.

\section{Projection onto $\Phi_0$}\label{sec:algorithmic-details}

To project onto the set $\Phi_0$, note that we can compute the projection one column at a time.  Moreover, the projection for each individual column amounts to (after permuting the entries of the column),
\[
\{\vct{x}\in\R^f~:~ 0\leq x_i \leq x_1  ~\forall i\,, x_1\leq 1\}\,.
\]
Assume, again without loss of generality, that we want to project a vector $\vct{z}$ with $z_2 \geq z_3 \geq \ldots \geq  z_n$.  Then we need to solve the quadratic program
\begin{equation}\label{eq:squishme}
\begin{array}{ll}
	\minimize & \tfrac{1}{2} \|\vct{z}-\vct{x}\|^2\\
	\st &  0\leq x_i \leq x_1  ~\forall i\,, x_1\leq 1
	\end{array}
\end{equation}
The optimal solution can be found as follows. Let $k_c$ be the largest $k \in \{2,\ldots, f\}$ such that 
\[
z_{k_c+1} \leq \Pi_{[0,1]}\left(\sum_{k=1}^{k_c}  z_k \right)=:\mu
\]
where $\Pi_{[0,1]}$ denotes the projection onto the interval $[0,1]$.  Set
\[
	\hat{x}_i = \begin{cases} \mu & i\leq k_c\\
		(z_i)_+ & i >k_c
	\end{cases}\,.
\]
Then $\hat{\vct{x}}$ is the optimal solution.  A linear time algorithm for computing $\hat{x}$ is given by Algorithm~\ref{alg:colsquish}

To prove that $\hat{\vct{x}}$ is optimal, define 
\[
	y_i = \begin{cases} 
		z_i - \mu & i \leq k_c\\
		\min(z_i,0) & i > k_c
	\end{cases}\,.
\]
$y_i$ is the gradient of $\tfrac{1}{2}\|\vct{x}-\vct{z}\|^2$ at $\hat{x}$.  Consider the LP
\[
\begin{array}{ll}
	\minimize & -\vct{y}^T \vct{x}\\
	\st &  0\leq x_i \leq x_1  ~\forall i\,, x_1\leq 1
	\end{array}\,.
\]
$\hat{\vct{x}}$ is an optimal solution for this LP because the cost is negative on the negative entries, $0$ on the nonnegative entries that are 
larger than $k_c$, positive for $2\leq k \leq k_c$,  and nonpositive for $k=1$. Hence, by the minimum principle, $\hat{\vct{x}}$ is also a solution of~\eq{squishme}.

\newpage

\begin{algorithm}[t]
  \caption{Column Squishing}
   \begin{algorithmic}[1]
  \REQUIRE A vector $\vct{z} \in \R^f$ with $z_2 \geq z_3 \geq \ldots \geq  z_n$.  
  \ENSURE The projection of $\vct{z}$ onto $\{\vct{x}\in\R^f~:~ 0\leq x_i \leq x_1  ~\forall i\,, x_1\leq 1\}$.
\STATE $\mu \leftarrow z_1$.
 \FOR{$k=2,\ldots,f$}
	\STATE {\bf if} $z_{k} \leq \Pi_{[0,1]}(\mu)$, Set $k_c=k-1$ and {\bf break}
  	\STATE {\bf else} set $\mu = \frac{k-1}{k}\mu + \tfrac{1}{k}z_k$.
\ENDFOR
\STATE $x_1 \leftarrow  \Pi_{[0,1]}(\mu)$ 
\STATE {\bf for} $k=2,\ldots, k_c$ {\bf set} $x_k =  \Pi_{[0,1]}(\mu)$.
\STATE {\bf for} $k=(k_c+1),\ldots, f$ {\bf set} $x_k =  (z_i)_+$.
\RETURN $\vct{x}$
\end{algorithmic} \label{alg:colsquish}
\end{algorithm}

\phantom{yada yada}

\end{document}